\documentclass[twocolumn,3p,sort&compress]{elsarticle}
\usepackage{common}

\usepackage[compact]{titlesec}  
\titlespacing{\section}{0pt}{4pt}{0pt}

\newcommand{\source}{\vec{y}}  
\newcommand{\pre}{\Gamma}  
\newcommand{\preinv}{\pre^{-1}}  
\newcommand{\A}{A}
\newcommand{\B}{B}  
\newcommand{\Op}{G}  
\newcommand{\EE}{\vec{x}}  
\newcommand{\residue}{\Delta}
\newcommand{\raw}[1]{#1_{0}}  
\newcommand{\bias}[1]{\overline{#1}}
\newcommand{\fixedpointfactor}{{\alpha}}  
\newcommand{\scale}{c}  
\newcommand{\Scale}{C}  
\newcommand{\condfactor}{S}  
\newcommand{\iterationcounter}{i}

\newcommand{\flux}{\vec{J}} 
\newcommand{\density}{u} 
\newcommand{\arate}{\eta} 
\newcommand{\vr}{\vec{r}}
\newcommand{\preconderr}{\norm{M}}
\newcommand{\lmin}{\lambda_\text{min}}
\newcommand{\lmax}{\lambda_\text{max}}

\definecolor{fast}{rgb}{0.0, 0.5,  0.0}
\newcommand{\fast}[1]{\color{fast} \textbf{#1}}
\newcommand{\minit}[1]{#1}

\begin{document}
	\begin{frontmatter}
	    \title{A universal matrix-free split preconditioner for\\ the fixed-point iterative solution of non-symmetric linear systems}
	    
	    \author[1]{Tom Vettenburg}
	    \address[1]{School of Science and Engineering, University of Dundee, Nethergate, Dundee, DD1\,4HN, Scotland, United Kingdom.}
	    \ead[1]{t.vettenburg@dundee.ac.uk}
	    \author[2]{Ivo M.~Vellekoop}
	    \address[2]{TechMed Centre, Biomedical Photonic Imaging, Faculty of Science and Technology, University of Twente, Enschede, The Netherlands.}
	    \ead[2]{i.m.vellekoop@utwente.nl}
	    \begin{abstract}
	        We present an efficient preconditioner for linear problems $A\EE=\source$.
	        It guarantees monotonic convergence of the memory-efficient fixed-point iteration for all accretive systems of the form $\A = L + V$, where $L$ is an approximation of $\A$, and the system is scaled so that the discrepancy is bounded with $\norm{V}<1$. In contrast to common splitting preconditioners, our approach is not restricted to any particular splitting. Therefore, the approximate problem can be chosen so that an analytic solution is available to efficiently evaluate the preconditioner. We prove that the only preconditioner with this property has the form $(L+\id)(\id - V)^{-1}$. 
	        This unique form moreover permits the elimination of the forward problem from the preconditioned system, often halving the time required per iteration.
	        We demonstrate and evaluate our approach for wave problems, diffusion problems, and pantograph delay differential equations. With the latter we show how the method extends to general, not necessarily accretive, linear systems.
	    \end{abstract}

    \end{frontmatter}
    
    \section{Introduction}
        \noindent Large linear systems are ubiquitous in the physical sciences and engineering. A vast body of work has been dedicated to the solution of the many problems that are simply too large to be represented as a coefficient matrix in computer memory. A sparse representation of the problem may offer reprieve, though typically at the cost of approximations such as finite differences~\cite{manteuffel1980incomplete, Bai2001, Zhang2002, Taflove05, Bai2009, saad2003iterative}. Matrix-free methods can avoid such approximations as they only require the computability of the forward problem. With closed-form expressions for the forward problem, matrix-free methods tend to be as efficient as they are versatile~\cite{Trefethen1997,saad2003iterative}.
        
        Perhaps the most memory-efficient iterative method for solving large systems is the fixed-point, or Richardson, iteration. Unfortunately this method often must be ruled out due to its tendency to diverge. Krylov-subspace methods are instead relied upon, even if these require the storage of multiple intermediate solution vectors. Widely-used and well-studied algorithms include the generalized minimum residuals (GMRES)~\cite{Saad1986}, the conjugate gradient squared (CGS)~\cite{Sonneveld89}, and the stabilised bi-conjugate gradient (BiCGSTAB)~\cite{vanderVorst1992,Gutknecht1993}.
        The efficiency, convergence, and stability of these algorithms often hinge on the availability of an appropriate preconditioner~\cite{vanderVorst2003,Wathen2015,Trefethen1997}. With the aid of a well-chosen preconditioner, $\pre$, a poorly-conditioned intangible problem, $\A\EE = \source$, could be solved efficiently after conversion into an equivalent, preconditioned, problem, $\preinv\A\EE = \preinv\source$. 

        Identifying an appropriate preconditioner is not always straightforward.
       
        Common matrix-free approaches to construct a preconditioner are based on the splitting of $\A$ into its Hermitian ($H$) and skew-Hermitian ($S$) parts~\cite{golub2000preconditioning, Bai2003,Arridge2013}, or shifting the original operator by a multiple of the identity operator, $\id$~\cite{Bai2006}. While this can sharply reduce the number of \emph{outer} iterations, the evaluation of the preconditioner, itself, involves the non-trivial inversion of large linear operators as $H$, $S+\id$, or $\A+\id$. This typically requires the use an iterative \emph{inner} algorithm in each step~\cite{cheung2002block, Bai2003, Arridge2013, Bai2006, Li2011}. The computational cost of such preconditioner limits its potential and may well outweigh its advantages~\cite{Bai2003, Bai2006}.
        
        Here, we adopt a alternative approach to construct a preconditioner. The proposed preconditioner is unique in the sense that it is the only form that ensures monotonic convergence of the fixed-point iteration with a general splitting $\A = L + V$, where $L$ is some approximation of the system and $V$ is bounded. As we demonstrated with several examples, a closed-form expression for $L$ can typically be found to evaluate the preconditioner efficiently without inner iterations. The proposed preconditioner further permits the elimination of the forward problem, rendering its computational cost negligible.
        
        This paper is structured as follows. First, we introduce our preconditioner and detail the implementation of the iterative algorithm, as well as the procedure to convert any linear system to the required form. In Section~\ref{sec:examples}, we demonstrate the preconditioning method with a variety of different linear problems: the Helmholtz problem, the diffusion equation, the pantograph delay-differential equation and for finding eigenmodes of the Schr\"odinger equation. In Section~\ref{sec:discussion}, we analyse the convergence behaviour of the universal split preconditioner and a shift-splitting preconditioner for different algorithms and problems. The ability to choose an arbitrary splitting with the universal-split preconditioner proved essential to ensure efficient convergence with GMRES, BiCGSTAB, and even the memory-efficient fixed-point iteration.
    
    \section{The universal split preconditioner.}
        \noindent We consider linear systems of the form $\A\EE = \source$ where $\EE$ and $\source$ are vectors of some Hilbert space $\cH$, and $\A$ is an invertible linear operator. For the moment, we assume that $\A$ is accretive, i.e.~that it has a non-negative real part, denoted as
        \begin{equation}
            \Re[\A]\geq 0 \;\; \text{ with }
            \Re[\A]\defeq \inf_{\EE\in D(\A)\setminus \vec{0}}\Re \frac{\inp{\EE,\A\EE}}{\inp{\EE, \EE}}\label{eq:accretivity-condition}
        \end{equation}
        with $D(\A)\subseteq\cH$ the domain of $\A$. Accretive operators, also called non-symmetric positive definite, or non-Hermitian positive semi-definite\cite{Bai2006}, include positive definite and skew-Hermitian operators, M-matrices, and positive stable matrices\cite{axelsson1996iterative}. Later we show how also non-accretive linear systems can be solved by converting them to an equivalent accretive form.
        
        We call a preconditioned system \emph{stable} if $M\defeq\id - \preinv\A$ is a contraction, meaning that $\norm{M}<1$, where $\norm{\cdot}$ denotes the operator norm. If the system is stable, the Neumann series 
            \begin{equation} \left(\sum_{\iterationcounter=0}^{\infty}M^\iterationcounter\right) \; \preinv\source = (\id - M)^{-1}\preinv \source = \EE,\label{eq:neumann_series}
        \end{equation}
        converges monotonically to the solution of $A\EE=\source$.
        
        Our main result is the construction of a preconditioner that satisfies the condition $\norm{\id-\preinv A}<1$ for arbitrary accretive operators $\A$.
         
        \begin{theorem}\label{th:maintext-main}
            Let $\A$ be an invertible accretive linear operator, and $L$ any linear operator such that $V \defeq A - L$ has $\norm{V}<1$. Then, it follows that the preconditioner
            \begin{equation}
                \pre \defeq (L+\id)(\id-V)^{-1}\fixedpointfactor^{-1} \label{eq:preconditioner}
            \end{equation}
            with $0<\fixedpointfactor<\fixedpointfactor_\text{max}$ is invertible, and ensures that $\norm{\id-\preinv\A} < 1$ , with $\fixedpointfactor_\text{max}\defeq 2\Re\left[\A^{-1}+(\id-V)^{-1}\right]> 1$.
            \begin{proof}
                The invertibility of $L+\id$ and $\pre$, as well as the contractivity of $\id - \preinv\A$ follows as a special case of Theorem~\ref{th:convergence-proof}, proven in the Appendix (see Corollary~\ref{th:main-corollary}). 
            \end{proof}
        \end{theorem}
        By Theorem~\ref{th:maintext-main}, our construction is universal in the sense that the resulting preconditioner stabilises \emph{any} accretive operator $\A$, regardless of how the splitting $\A=L+V$ is chosen. Hence, a readily-invertible $L+\id$ can be selected to avoid the two-level iterations that are typically required with a specific splitting~\cite{Bai2003,Bai2006,Arridge2013}. Furthermore, $\preinv$ is \emph{first order} in $(L+\id)^{-1}$, which is often the operation that dominates the calculation time. In Theorem~\ref{th:uniqueness} (Appendix), we show that it is the \emph{only} preconditioner with these two properties.
    
        The preconditioner places a bound $\kappa \approx 2 \norm{\A^{-1}}\norm{V}$ on the condition number (Theorem~\ref{th:convergence_rate_bounds}, Appendix), even if the original system $\A$ is unbounded, as indeed many physical systems are.
        In other words, any sufficiently close approximation, $L = A - V$, will improve the convergence rate with any iterative method.

    \subsection{Implementation.}
        \noindent A distinctive feature of our preconditioner is that the preconditioned system $\preinv\A$ can be simplified to eliminate the forward operator $\A$. Defining $\B \defeq \id - V$, we can substitute $\A = L+\id - \B$, giving
        \begin{align}
            \preinv A &= \fixedpointfactor \B\left[\id-(L+\id)^{-1}\B\right].
        \end{align}
        This keeps the preconditioning cost to a minimum. The elimination of the potentially unbounded operator $\A$ also facilitates adequate sampling prior to its numerical evaluation.
    
        Our preconditioner can be used in many iterative schemes, including the memory-efficient fixed-point iteration, $\EE\rightarrow\EE+M\EE$, producing Neumann series~\eqref{eq:neumann_series}. Compared to other iterative schemes, the fixed-point iteration uses a minimal amount of temporary storage, a single vector, and does not require the computation of inner products. Monotonic convergence is only ensured when $\norm{M}<1$, which is precisely what the proposed preconditioner achieves. As illustrated in Algorithm~\ref{alg:basic}, the algorithm starts with an initial estimate, e.g.~$\vec{0}$, and repeatedly applies operator $M$ to find increasingly accurate estimates of $\EE$. Here, $\fixedpointfactor\in(0,1]$ can be interpreted as the step size of the fixed-point iteration (see Appendix for a discussion on the optimal step size), and $\Delta= \preinv (\source - A\EE)$ as the residual of the preconditioned system. The iteration can be stopped when the relative updates are less than a predefined threshold, $\epsilon_\mathrm{max}$.
        \begin{algorithm}[H]
            \caption{Solve the accretive system $\A\EE = \source$ for $\EE$.}\label{alg:basic}
            \begin{algorithmic}[1]
                \State $\EE \gets \vec{0}$ \Comment{ initial estimate}
                \Repeat \label{op:iteration_begin}
                \State $\residue \gets \B\left[\left(L+\id\right)^{-1} \left(\B\EE + \source\right) - \EE\right]$\label{op:residue} \Comment{ residual}
                \State $\EE \gets \EE + \fixedpointfactor \residue$ \Comment{ update}
                \Until{$\norm{\residue}/\norm{\preinv\source} < \epsilon_\mathrm{max}$} \label{op:iteration_end} \Comment{ convergence criterion}
            \end{algorithmic}
        \end{algorithm}
    
    \subsection{Conversion to canonical form.}
        \noindent We can show that any invertible linear problem can be brought into the canonical form $\A\EE=\source$ where $\A$ is accretive and $\A=L+V$ with $\norm{V}<1$. Here, we use a simple two-step heuristic to find such a splitting. Starting with the original problem $\raw{\A}\raw{\EE}=\raw{\source}$, we first identify an approximate linear system $\raw{L}$ such that the norm $\norm{\raw{\A}-\raw{L}}$ is minimised. Next, we divide the original problem by a sufficiently large complex scalar, $\scale$, so that $\norm{V} < 1$. The  canonical system is then given by
        \begin{align}
            \A &\defeq \scale^{-1} \raw{\A} \qquad& \EE &\defeq \raw{\EE}\\
            L &\defeq \scale^{-1} \raw{L} \qquad& \source &\defeq \scale^{-1} \raw{\source},
        \end{align}
        and $V\defeq \A - L$.
        
        The optimal value of $\norm{V}$ depends on the properties of the system (see Appendix). However, we found that the convergence rate is not very sensitive to the choice of $\norm{V}$, so we simply use $\norm{V}=0.95$ in what follows. When the numerical range of $\raw{\A}$ is confined to some half of the complex plane, the complex argument of $\scale$ can always be chosen so that $\A = \frac1\scale\raw{\A}$ is accretive.
        In contrast, when the numerical range of the original problem $\raw{\A}\raw{\EE} = \left(\raw{L}+\raw{V}\right)\raw{\EE} = \raw{\source}$ is not confined to one half of the complex plane, a larger system, $\A$, can always be constructed from $\raw{\A}$ and its adjoint $\raw{\A}^\ast$ as
        \begin{align}
            A &\defeq
            \scale^{-1}
            \begin{bmatrix}
                0 & -\raw{\A}^\ast\\
                \raw{\A} & 0
            \end{bmatrix}
            \qquad &
            \EE&\defeq
            \begin{bmatrix}
                \raw{\EE}\\
                \raw{\EE}^\prime
            \end{bmatrix} \label{eq:antisymmetrised}
            \\
            L &\defeq \scale^{-1}\begin{bmatrix}
                0 & -\raw{L}^\ast\\
                \raw{L} & 0
            \end{bmatrix} \qquad& 
            \source&\defeq
            \scale^{-1}
            \begin{bmatrix}
                -\raw{\source}^\prime\\
                \raw{\source}
            \end{bmatrix}.
        \end{align}
        Here, the adjoint of $\raw{L}$ is denoted by $\raw{L}^\ast$, while $\raw{\EE}^\prime$ and $\raw{\source}^\prime$ are auxiliary vectors. By choosing a real scaling factor, $c$, this construction of $\A$ is skew-Hermitian and naturally fulfils the accretivity requirement of Eq.~\eqref{eq:accretivity-condition}.

        It can be noted that the antisymmetrised problem of Eq.~\eqref{eq:antisymmetrised} combines the original problem and the adjoint problem $\raw{\A^\ast}\raw{\EE}^\prime = \raw{\source}^\prime$. If one is not interested in simultaneously solving an adjoint problem, we can choose $\raw{\source}^\prime=0$ and ignore $\EE^\prime \to 0$.
        
        \begin{figure*}
            \centering
            \includegraphics[width=\linewidth]{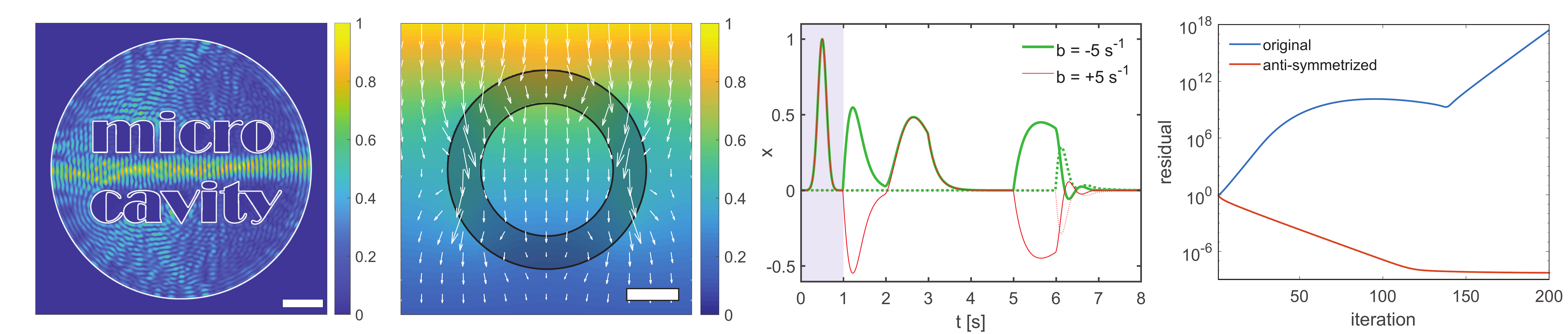}
            \begin{flushleft}
                \vspace*{-18pt}
                \hspace*{2pt}
                \begin{subfigure}[b]{37mm}
                    \caption{}\label{fig:helmholtz}
                \end{subfigure}
                \begin{subfigure}[b]{40mm}
                    \caption{}\label{fig:diffusion}
                \end{subfigure}
                \begin{subfigure}[b]{44mm}
                    \caption{}\label{fig:pantograph}
                \end{subfigure}
                \begin{subfigure}[b]{8mm}
                    \caption{}\label{fig:convergence_divergence}
                \end{subfigure}
            \end{flushleft}
            \caption{Preconditioned solutions to numerical problems in three distinct classes. \textbf{\subref{fig:helmholtz}} Scattering of a wave in an air-filled microcavity. Boundaries and letters (indicated by the white outline) have a refractive index of $2.8954 + 2.9179\ii$ (iron), and the source is located in a ring just inside the cavity wall. Scale bar: \SI{2.5}{\micro\meter}. \textbf{\subref{fig:diffusion}} Diffusion in a medium with an anisotropic diffusion coefficient, ($1\;\text{mm}^{-2}s^{-1}$ radially and $25\;\text{mm}^{-2}s^{-1}$ tangentially) in a ring (shaded area), in a background of $2\;\text{mm}^{-1}$. Sources and sinks are placed at the top and bottom, respectively. Scale bar: \SI{10}{\milli\meter}. \textbf{\subref{fig:pantograph}} The different responses of two pantograph delay-differential systems to the same Gaussian pulse. The light blue background marks the initial condition.
            \textbf{\subref{fig:convergence_divergence}} Divergence without and convergence with preconditioner. The relative residue is shown as a function of the iteration number for the Helmholtz problem (Fig.~\ref{fig:convergence_divergence}). The simulation was terminated at a relative residue of $<1\cdot 10^{-6}$}
            \label{fig:overview_figure}
        \end{figure*}

	\section{Results}\label{sec:examples}
    	\noindent In this section, we demonstrate how our preconditioner can be applied to linear problems encountered in science and engineering. As examples, we choose hyperbolic and parabolic partial differential equations (the wave equation and the diffusion equation, respectively), a delay-differential equation (the pantograph equation), and investigate the application of our preconditioner to finding the lowest energy eigenmodes of the Schr\"odinger equation. A quantitative comparison of different preconditioners and iterative algorithms will be given in Section~\ref{sec:discussion}.  
    	
    \subsection{Wave equations.}
        \noindent The equations describing acoustic, seismic, electromagnetic, and gravitational waves, or quantum wavefunctions are often linear and stationary to a very good approximation. We previously developed highly efficient approaches to solve the scalar Helmholtz equation~\cite{Osnabrugge2016}, and the time-harmonic Maxwell's equations in arbitrary anisotropic dielectric and magnetic materials~\cite{Vettenburg2019}. Other authors have demonstrated its efficiency for isotropic media, as well as for sound and seismic waves~\cite{Krueger2017,Kaushik2020,Huang2020}. We will now demonstrate that these approaches can be derived, analysed, and improved, within the proposed common framework.
    
        The inhomogeneous Helmholtz equation can be stated as
        \begin{align}
            \nabla^2\psi(\vr) + k^2(\vr) \psi(\vr) = -S(\vr)\label{eq:helmholtz},
        \end{align}
        where the physical field is represented by $\psi(\vr)$, a scalar function of the position, $\vr$.
        The position-variant wavenumber, $k(\vr)$, is generally a complex function. Provided that the system is gain-free, i.e.~the imaginary part $\Im[k^2(\vr)]\geq 0$, the numerical range of the Helmholtz equations will be confined to the upper half of the complex plane.
        
        Following the procedure outlined above, we first choose an appropriate centre value $\bias{k^2} \in \bbC$ so that $\norm{k^2(\vr) - \bias{k^2}}$ is minimised, an optimisation problem that is known as the smallest circle problem~\cite{Megiddo83}. Then, we divide through the complete system with a factor $\scale = - \frac1{0.95\ii} \norm{k^2(\vr)-\bias{k^2}}$ such that $\norm{V}=0.95$ and the resulting system is accretive.
        The scaled operator $L+\id = \frac1c \nabla^2 + \frac1c \bias{k^2} + \id$ is translation invariant and therefore can efficiently be inverted using a Fourier transform
        \begin{equation}
            (L+\id)^{-1}=\cF^{-1}\frac{\scale}{-\norm{\vec{p}}^2 + \bias{k^2} + \scale}\cF.\label{eq:laplacian_inverse}
        \end{equation}
        Here, $\cF$ denotes the Fourier transform over all spatial dimensions, and $\vec{p}$ is the Fourier-space coordinate vector. On a regular grid, a fast Fourier transform can be used to evaluate $(L+\id)^{-1}$ efficiently. To avoid the implicit periodic boundary conditions, absorbing boundaries can be added. Their thickness can be minimised using a modified fast Fourier transform that eliminates first-order wrap-around artefacts~\cite{Osnabrugge2021}.
    
        This result improves on our prior work~\cite{Osnabrugge2016} in several ways. Firstly, it is no longer required for $L$ and $V$ to be accretive individually. Therefore, we can choose $\bias{k^2}$ have a positive imaginary part if that reduces $\norm{V}$. As we will show in Section~\ref{sec:discussion}, this choice speeds up convergence in systems where absorption dominates.
        
        Secondly, the earlier work implicitly used a fixed-point iteration with $\norm{V} \rightarrow 1$ and $\fixedpointfactor = 1$. In the new framework, it is clear that any combination of $\norm{V} < 1$ and $\fixedpointfactor < 1$ can be chosen to maximise the convergence rate. Finally, the same approach can be used with other widely-used algorithms such as GMRES and BiCGSTAB.
        
        We first confirmed that our Helmholtz solver converges to the analytical solution for propagation through an empty medium~\cite{github_anysim}. Next, we tested a 1-D geometry (a glass plate, $n=1.5$, in vacuum), and a 2-D geometry (an iron structure, $n=2.8954 + 2.9179i$, irradiated from the surrounding ring). This structure represents a challenging test case because of the high refractive index contrast (causing $\norm{V}$ to be large) and the cavity nature of the structure (causing a large $\norm{\A^{-1}}$), both factors contributing to a large condition number. Figure~\ref{fig:helmholtz} shows the solution for the 2-D geometry on a $480\times 480$ grid after 6026 iterations of the preconditioned fixed-point iteration. Without preconditioning, neither the fixed-point iteration nor alternative methods (GMRES and BiCGSTAB) converged to a solution in this number of iterations.
        
        As a final note, more general problems, such as solving the time-harmonic Maxwell's equations with anisotropic, chiral, and magnetic materials---including those with negative refractive index, can be addressed by similarly inverting the vector Laplacian and incorporating the heterogeneity in a matrix-valued function $V(\vr)$~\cite{Vettenburg2019}. In this prior work it was hypothesised that the accretivity of both $L$ and $V$, individually, is sufficient to ensure monotonic convergence for all materials. With Theorem~~\ref{th:maintext-main}, we prove that this hypothesis indeed holds true, and that convergence is even guaranteed for the less strict condition that the sum $\A = L + V$ is accretive.

    \subsection{Diffusion equations.}\label{sec:diffusion}
        \noindent In this section, we consider the heat equation / diffusion equation of the form~\cite{Wang2012,Carslaw1959book}
        \begin{equation}
            \partial_t \density(\vr,t) = \nabla \cdot \left[D(\vr,t) \nabla \density(\vr,t)\right] - \arate(\vr,t) \density(\vr,t) + S(\vr,t)
            \label{eq:diffusion}
        \end{equation}
        where $\density(\vr, t)$ is the unknown density distribution to solve for, $\arate\;[s^{-1}]$ is the absorption rate, $S$ represents the source and sink density, and $\partial_t$ is the partial derivative with respect to time $t$. In 3-D space, the diffusion coefficient $D\left[m^2s^{-1}\right]$ is a $3 \times 3$-matrix-function. It need not be positive definite or satisfy the Onsager reciprocity relations~\cite{Chen2018negative},
        it is sufficient that it is accretive and invertible.
    
        As with the Helmholtz problem, we would like to incorporate the differential operators in $L$, and leave any spatial variation in $D$ or $\arate$ for $V$. However, in Eq.~\eqref{eq:diffusion}, the product of $D(\vr, t)$ and two differential operators hampers rewriting it as a sum $L+V$. To proceed, we define the flux $\flux(\vr,t) \defeq -D(\vr,t) \nabla \density(\vr,t)$ and split the diffusion equation into Ficks's two laws of diffusion\cite{Wang2012}
        \begin{align}
            \partial_t \density + \nabla \cdot \flux + \arate \density &= S\qquad\text{and}\label{eq:diffusion-F-scalar}\\
            D^{-1}\flux + \nabla u &= 0,\label{eq:diffusion-F-vector}
        \end{align}
        where we omitted the $(\vr,t)$-dependency for brevity, and left-multiplied the last equation by $D^{-1}$ to separate the differential operator $\nabla$ and the spatially varying $D$.
    
        Eqs.~\eqref{eq:diffusion-F-scalar} and \eqref{eq:diffusion-F-vector} can now be written in matrix form $(L+V)\EE = \source$ using the definitions
        \begin{align}
            L &\defeq \Scale^{-\frac12}
            \begin{bmatrix}
                \partial_t + \bias{\arate} & \nabla^T\\
                \nabla & \bias{D^{-1}}
            \end{bmatrix} \Scale^{-\frac12}
            &
            \EE &\defeq \Scale^{\frac12}
            \begin{bmatrix}
                u\\ \flux
            \end{bmatrix}
            \\
            V &\defeq
            \Scale^{-\frac12}
            \begin{bmatrix}
                \arate - \bias{\arate} & 0\\
                0 & D^{-1} - \bias{D^{-1}}
            \end{bmatrix}
            \Scale^{-\frac12}
            &
            \source &\defeq \Scale^{-\frac12}
            \begin{bmatrix}
                S\\
                \mathbf{0}
            \end{bmatrix},\label{eq:diffusion-equation-matrix-form}
        \end{align}
        where $\Scale$ is a positive definite diagonal matrix that, as the scaling factor $\scale$, ensures that $\norm{V}=0.95$. Observe that the operator $L+V$ is a sum of skew-Hermitian operators ($\nabla$ and $\partial_t$) and accretive operators ($\arate$ and $D^{-1}$). Therefore, $L+V$, is accretive, and our preconditioner can be used.
        To choose $\Scale$, we first solved the smallest circle problem for each individual element of the tensor field $\raw{V}$ and included the centres of each circle (denoted as the scalar $\bias{\arate}$ and matrix $\bias{D^{-1}}$) as an offset in $\raw{L}$. We stored the radii of the circles in a separate matrix $\delta V$, and we designed $\Scale$ to equilibrate $\delta V$ using the algorithm described in~\cite{duff2001algorithms}. Importantly, this equilibration step makes $V$ non-dimensional, and reduces the amount of required rescaling which, in turn, improves the convergence rate (see Appendix). Finally, an overall scaling was applied to ensure $\norm{V}=0.95$.

        We first confirmed that our solver converges to the analytical solution for the diffusion of light in a slab geometry\cite{github_anysim}. Outside the slab, we chose $\arate = D/z_e^2$, so that the steady-state solution to the diffusion equation obeys the mixed boundary condition
        $\density = -z_e \left(\vec{n}\cdot\nabla \density\right) \label{eq:diffusion-BE}
        $, with $z_e$ the extrapolation length of the boundary~\cite{Lagendijk1989,Zhu1991}.
    
        Fig.~\ref{fig:diffusion} shows an example of a diffusion calculation in a medium with a tensor diffusion coefficient on a $256\times 256$ grid. Inside the ring, the diffusion coefficient is \SI{25}{\milli\meter^2\second^{-1}} in the tangential direction and \SI{1}{\milli\meter^2\second^{-1}} in the radial direction. The background diffusion coefficient is homogeneous with \SI{2}{\milli\meter^2\second^{-1}}; the source is located at the top of the simulation window, and an absorbing layer is placed at the bottom.

    \subsection{Pantograph equation with variable coefficients.}\label{sec:pantograph}
        \noindent The use of our preconditioner is not restricted to partial differential equations. It can be applied to completely different classes of linear problems. For instance, in delay differential equations, the derivative of the function not only depends on its value at the current time but also on its past or future value. The delay can even vary in time as with the so-called pantograph equation~\cite{Ockendon1971}, which can be written as
        \begin{align}
            -\partial_t x(t) & = a(t) x(t) + b(t) x(\lambda t) & &\text{for } t \geq t_0,
            \label{eq:pantograph}
        \end{align}
        with starting condition $x(t) = x_0(t)$ for $t<t_0$, and $\lambda>0$ a constant coefficient. As an extension to the original formulation, we allow the dependence on future values, $\lambda > 1$, as well as coefficients $a$ and $b$ that are time-dependent and complex.
        
        We start by writing the equation in the form
        \begin{align}
            A\EE & =
            \scale^{-1} \left[\left(\partial_t + a(t)\right)x(t) + b(t)x(\lambda t)\right]\label{eq:pantograph-A}\\
            \source & = \scale^{-1} \left[b(t)x_0(\lambda t) + x_0(t_0) \delta(t-t_0)\right]
        \end{align}
        on the space $t\in[t_0,\infty)$ with the usual inner product. The boundary condition $x(t_0)=x_0(t_0)$ was converted to a Dirac-delta source at $t_0$. This way, $x(t)$ vanishes at the boundaries, making $\partial_t$ a skew-Hermitian operator.
        We choose $\raw{L} = \partial_t + \bias{a}$ as approximate system, where $\bias{a}$ is a complex scalar chosen to minimise $\sup \abs{a(t)-\bias{a}}$. It can be seen that the scalar $\scale = \frac1{0.95}\left[\sup\abs{a(t)-\bias{a}} + \sup\abs{b(t)}\right]$ then guarantees that $\norm{V}\leq 0.95$. 

        Fig.~\ref{fig:pantograph} shows the result for solving the pantograph equation for inhomogeneous $a$ and $b$. In this example, $\lambda=0.5$, and $x(t)=e^{-50 (t-1)^2}$ at times $t<t_0=\SI{1}{\second}$, and a grid spacing of $\SI{0.01}{\second}$ was used. We chose $a=\SI{5}{\second^{-1}}$ for the interval $t\in[\SI{1}{\second}, \SI{6}{\second}]$, and $5 - 10\ii$ after $\SI{6}{\second}$.
        The equation was solved for both $b=\SI{5}{\second^{-1}}$ and for $b=\SI{-5}{\second^{-1}}$. In each case, we set $b=0$ in the region $t\in[\SI{3}{\second}, \SI{5}{\second}]$ so that the solution decays exponentially in this time interval.
        
        Depending on the choice of parameters, the pantograph equation may not be accretive. For example, with a choice of $\lambda=0.9, a=0.1, b=-5$, the system is non-accretive. When basing the preconditioner directly on Eq.~\eqref{eq:pantograph-A}, the fixed-point iteration diverges (see Fig.~\ref{fig:convergence_divergence}).
        However, with the antisymmetrised system of Eq.~\eqref{eq:antisymmetrised}, monotonic convergence is guaranteed, and the iteration converges to a residue of below $10^{-8}$ in $125$ iterations for the same parameters before stagnating due to the finite machine precision (Fig.~\ref{fig:convergence_divergence}).

    \subsection{Lowest energy eigenmodes.}
        \noindent A common problem in physics and engineering is to find the eigenvectors of a linear system. The solution of interest are typically the ground state and the first excited states, i.e.~the eigenvectors, $\psi$, that correspond to the eigenvalues, $E$, of lowest magnitude in the eigenvalue equation $A\psi=E\psi$.
        
        \begin{figure}
            \centering
            \includegraphics[width=\linewidth]{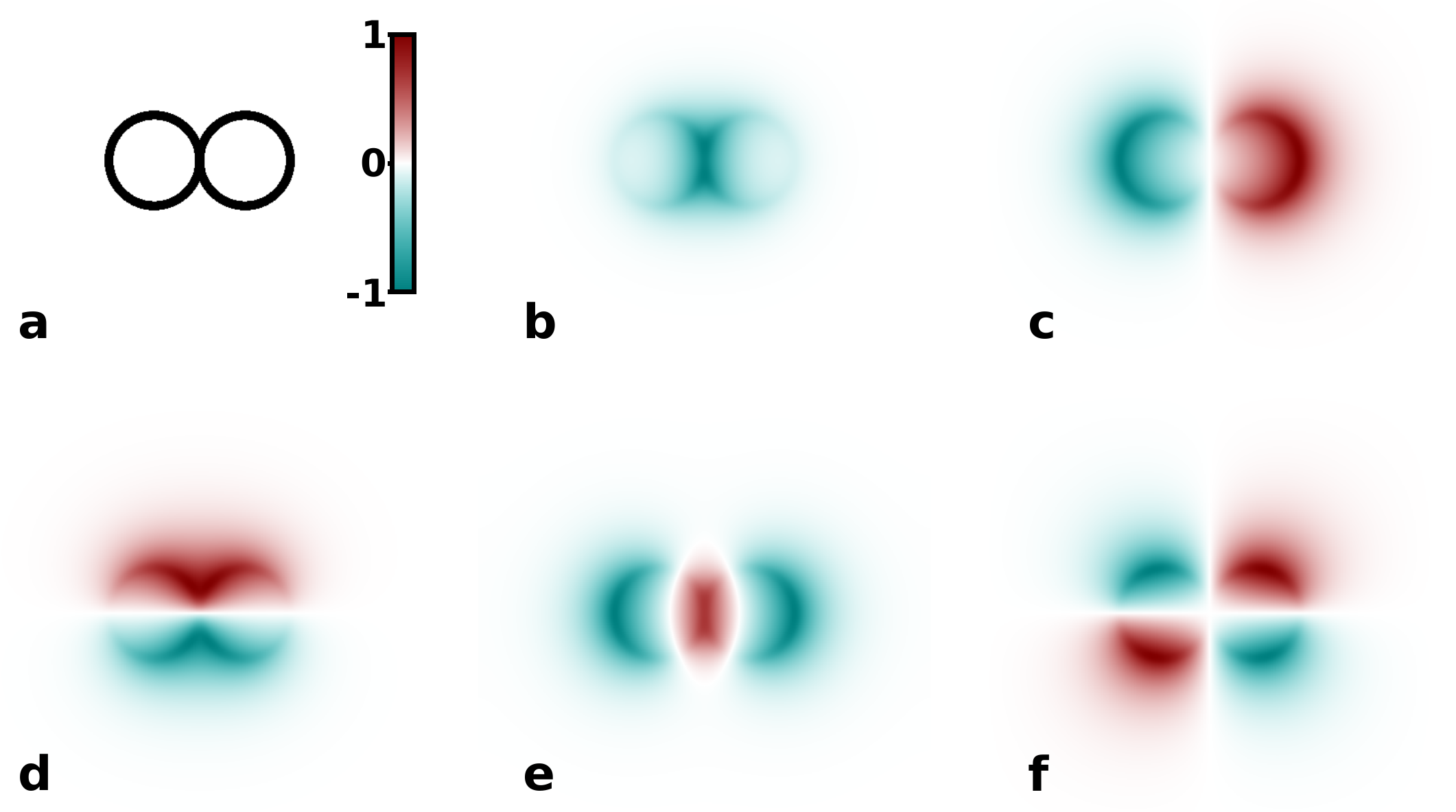}
            \begin{flushleft}
                \vspace*{-30pt}
                \hspace*{1pt}
                \begin{subfigure}[b]{10mm}
                    \phantomcaption{}\label{fig:eigenmodes_harmonic_a}
                \end{subfigure}
                \begin{subfigure}[b]{10mm}
                    \phantomcaption{}\label{fig:eigenmodes_harmonic_b}
                \end{subfigure}
                \begin{subfigure}[b]{10mm}
                    \phantomcaption{}\label{fig:eigenmodes_harmonic_c}
                \end{subfigure}
                \begin{subfigure}[b]{10mm}
                    \phantomcaption{}\label{fig:eigenmodes_harmonic_d}
                \end{subfigure}
                \begin{subfigure}[b]{10mm}
                    \phantomcaption{}\label{fig:eigenmodes_harmonic_e}
                \end{subfigure}
                \begin{subfigure}[b]{10mm}
                    \phantomcaption{}\label{fig:eigenmodes_harmonic_f}
                \end{subfigure}
            \end{flushleft}
            \caption{\textbf{Resonances in a double-ring resonator potential well.} \textbf{\subref{fig:eigenmodes_harmonic_a}} The potential well (black) has a depth of $10$ in natural units and with $m=1$. Each ring has unit radius and a channel width of $0.2$ units.  \textbf{\subref{fig:eigenmodes_harmonic_b}}-\textbf{\subref{fig:eigenmodes_harmonic_f}} The first $5$ eigenmodes in order of increasing eigenvalue are $E=7.5$, $8.5$, $8.5$, $8.7$, and $9.2$.}\label{fig:eigenmodes_harmonic}
        \end{figure}
    
        \begin{table*}[t]
        \small
        \centering
\begin{tabular}{l|c|c|c|c|c|c|c} \textbf{a. NO PRECONDITIONER} & GMRES20& GMRES5& BiCGSTAB& FP100& FP90& FP80& FP70\\
\hline
diffusion isotropic (homogeneous slab)  & m & m & d & d & d & d & d\\
diffusion anisotropic (Fig.~\ref{fig:diffusion})  &    6.8 k &    9.8 k & d & d & d & d & d\\
Helmholtz 1-D (glass plate, $n=1.5$)  &    1.2 k &    2.6 k &    3.7 k & d & d & d & d\\
Helmholtz 2-D $\bbR$ bias (Fig.~\ref{fig:helmholtz})  &    9.5 k &   26.7 k & d & d & d & d & d\\
Helmholtz 2-D $\bbC$ bias (Fig.~\ref{fig:helmholtz})  &    9.5 k &   26.3 k & d & d & d & d & d\\
Helmholtz 2-D $\bbR$ bias, dielectric  &    3.2 k &    7.1 k & s & d & d & d & d\\
Helmholtz 2-D $\bbC$ bias, dielectric  &    3.1 k &    7.1 k & s & d & d & d & d\\
pantograph non-symmetrised (Fig.~\ref{fig:pantograph})  & m & m &    3.0 k & d & d & d & d\vspace{1em}\\

\textbf{b. UNIVERSAL SPLIT PRECONDITIONER} & GMRES20& GMRES5& BiCGSTAB& FP100& FP90& FP80& FP70\\
\hline
diffusion isotropic (homogeneous slab)  & \minit{49} & 149 & \fast{60} & 578 & 642 & 722 & 826\\
diffusion anisotropic (Fig.~\ref{fig:diffusion})  & 86 & 248 & \fast{\minit{68}} & 371 & 412 & 464 & 530\\
Helmholtz 1-D (glass plate, $n=1.5$)  & 305 & \minit{300} & 430 & 463 & 323 & \fast{305} & 314\\
Helmholtz 2-D $\bbR$ bias (Fig.~\ref{fig:helmholtz})  & \minit{   3.2 k} &    4.7 k & \fast{   3.5 k} &   11.0 k &   12.1 k &   13.5 k &   15.4 k\\
Helmholtz 2-D $\bbC$ bias (Fig.~\ref{fig:helmholtz})  & \minit{   2.8 k} &    4.5 k & \fast{   3.4 k} &   29.5 k &    8.4 k &    7.7 k &    8.4 k\\
Helmholtz 2-D $\bbR$ bias, dielectric  & 125 & 142 & \minit{122} & 196 & \fast{129} & 132 & 146\\
Helmholtz 2-D $\bbC$ bias, dielectric  & 124 & 140 & \minit{121} & 173 & \fast{127} & 132 & 146\\
pantograph non-symmetrised (Fig.~\ref{fig:pantograph})  & \minit{13} & 17 & 18 & 88 & 23 & \fast{26} & 30\\
\end{tabular}
            \caption{The number of iterations (lower is better) for various problems, comparing the GMRES method (restarts after $20$ and $5$ iterations, respectively), BiCGSTAB, and fixed-point iteration with step size $\fixedpointfactor=1.0, 0.9, 0.8,$ and $0.7$, respectively.
            \textbf{a.}~The number of iterative evaluations of the systems without preconditioning. Failure to converge is indicated with the letter (s) for stagnation, (m) in case that the maximum number of iterations (30,000) are reached, and (d) when the iteration diverged.
            \textbf{b.}~The number of iterations when using the universal split preconditioner.
            All methods converge and those with the lowest execution time are highlighted in boldface green. Note that although GMRES tends to have fewer iterations, its overheads negate its advantage.}
            \label{tab:results_universal}
        \end{table*}
        
        \renewcommand{\fast}[1]{#1}  
        \begin{table*}[t]
        \small
            \centering
\begin{tabular}{l|c|c|c|c|c|c|c}
\textbf{a. SHIFT (OUTER ITERATIONS)} & GMRES20& GMRES5& BiCGSTAB& FP100& FP90& FP80& FP70\\
\hline
diffusion isotropic (homogeneous slab)  & s & \fast{\minit{50}} & s & i & 195 & 210 & 240\\
diffusion anisotropic (Fig.~\ref{fig:diffusion})  & \fast{\minit{28}} & 52 & 36 &    3.8 k & 123 & 138 & i\\
Helmholtz 1-D (glass plate, $n=1.5$)  & 97 & 106 & s & i & 67 & \fast{\minit{65}} & 75\\
Helmholtz 2-D $\bbR$ bias (Fig.~\ref{fig:helmholtz})  & s & \minit{   2.7 k} & s & m & \fast{   3.5 k} &    3.6 k &    3.9 k\\
Helmholtz 2-D $\bbC$ bias (Fig.~\ref{fig:helmholtz})  & s & \fast{\minit{   1.4 k}} & s & i &    1.6 k &    1.7 k &    1.9 k\\
Helmholtz 2-D $\bbR$ bias, dielectric  & s & 48 & 41 & i & 67 & \fast{\minit{31}} & 38\\
Helmholtz 2-D $\bbC$ bias, dielectric  & 35 & 41 & 38 & i & 37 & \fast{\minit{32}} & 36\\
pantograph non-symmetrised (Fig.~\ref{fig:pantograph})  & 9 & 10 & \fast{9} &    1.8 k & 27 & 13 & \minit{8}\vspace{1em}\\

\textbf{b. SHIFT (INNER ITERATIONS)}& GMRES20& GMRES5& BiCGSTAB& FP100& FP90& FP80& FP70\\
\hline
diffusion isotropic (homogeneous slab)  & s & \fast{\minit{   1.6 k}} & s & i &   12.3 k &   17.4 k &   23.3 k\\
diffusion anisotropic (Fig.~\ref{fig:diffusion})  & \fast{\minit{801}} &    2.1 k &    3.0 k &   92.4 k &   95.4 k &   98.8 k & i\\
Helmholtz 1-D (glass plate, $n=1.5$)  & \minit{   2.8 k} &    5.3 k & s & i &    0.7 M & \fast{   0.7 M} &    0.7 M\\
Helmholtz 2-D $\bbR$ bias (Fig.~\ref{fig:helmholtz})  & s & \minit{   0.3 M} & s & m & \fast{   1.1 M} &    1.2 M &    1.2 M\\
Helmholtz 2-D $\bbC$ bias (Fig.~\ref{fig:helmholtz})  & s & \fast{\minit{   0.2 M}} & s & i &    1.1 M &    1.1 M &    1.2 M\\
Helmholtz 2-D $\bbR$ bias, dielectric  & s & \minit{   8.4 k} &   12.6 k & i &    0.3 M & \fast{   0.3 M} &    0.3 M\\
Helmholtz 2-D $\bbC$ bias, dielectric  & \minit{   3.2 k} &    6.3 k &    9.7 k & i &    0.4 M & \fast{   0.4 M} &    0.4 M\\
pantograph non-symmetrised (Fig.~\ref{fig:pantograph})  & \minit{92} & 184 & \fast{263} &   20.1 k &   20.4 k &   20.5 k &   20.6 k\\
\end{tabular}
            \caption{Evaluation of the shift-splitting preconditioner for the same problems and algorithms investigated in Table~\ref{tab:results_universal}. \textbf{a.} The number of outer iterations needed to solve the system with the shift-splitting preconditioner. \textbf{b.} The total number of evaluations of the forward problem. This is orders of magnitude higher than the number of outer iterations since each evaluation of the shift-splitting preconditioner requires itself the iterative solution of a shifted linear system. Although the shift-splitting preconditioner ensures convergence in a low number of iterations, its total number of operations, and thus run-time, exceeds that of the universal split preconditioner (Table~\ref{tab:results_universal}).
            }
            \label{tab:results_shift}
        \end{table*}
        The eigenvectors can be found using Lanczos' method for symmetric systems or the Arnoldi iteration for non-symmetric systems. However, determining the lowest energy solutions involves the repeated inversion of the corresponding linear problem~\cite{Trefethen1997}. An appropriate preconditioner is thus also important to solve large eigenvector problems efficiently.
        
        Although an eigenvalue problem is not necessarily accretive, it is sufficient that its spectrum is bounded from below. Adding a constant value, $\bias{\A}$, to both sides of the eigenvalue equation changes the eigenvalues correspondingly while leaving its eigenvectors untouched. Our preconditioner can thus be used to calculate the lowest energy solutions for arbitrary bounded potentials.
        
        As an example, we study the motion of a simple quantum-mechanical particle with mass, $m$, in a spatially-variant potential, $V_{s}$. The time-independent Schr\"odinger equation is then given by:
        \begin{align}
            \left(-\frac{\hbar^2}{2m}\nabla^2 + V_{s} + \bias{\A} \right)\psi = \left(E + \bias{\A}\right)\psi. \label{eq:schrodinger}
        \end{align}
        Without loss of generality, in what follows we use a unit mass and natural units so that $m=\hbar=1$.
        Fig.~\ref{fig:eigenmodes_harmonic_a} depicts a potential well of depth $10$ units, consisting of two resonating rings of diameter $2$ and width $0.2$. Fig.~\ref{fig:eigenmodes_harmonic_b}-\ref{fig:eigenmodes_harmonic_f} shows the lowest energy eigenmodes. The five modes in this structure are calculated on a $512 \times 512$ grid using SciPy's implicitly restarted Lanczos method for Hermitian systems. In each iteration, Eq.~\eqref{eq:schrodinger} was inverted using the conjugate gradient method. Preconditioning reduced the condition number from $295.3$ to $1.77$, a $166$-fold improvement.

    \section{Results and discussion}\label{sec:discussion}
    \subsection{Evaluation with different iterative methods}
        \noindent The universal split preconditioner guarantees the convergence of the fixed-point iteration (Alg.~\ref{alg:basic}) for all splittings. Since it also significantly improves the condition number, we hypothesised that it also accelerates the convergence of other iterative methods. To tests this hypothesis, we solved eight different problems using the commonly-used GMRES method, the BiCGSTAB algorithm, and the fixed-point iteration with different values of the step size $\fixedpointfactor$. For each problem and algorithm, we compared the convergence behaviour before and after preconditioning with the universal split preconditioner. 
        
        All tests were performed using single-precision arithmetic on an NVIDIA RTX 3050 Ti GPU, with the default GPU-accelerated implementations of GMRES and BiCGSTAB in MATLAB\;2021b~\cite{github_anysim}. The tests correspond to the situations discussed in Section~\ref{sec:examples}. For the Helmholtz problem, we also investigated a dielectric structure with the same structure as in Fig.~\ref{fig:helmholtz}, but with refractive indices of 1.33 and 1.46, typical values for biological tissue. In addition, we compared the use of a real bias (as done for the Helmholtz equation previously~\cite{Osnabrugge2016}) and a complex-valued bias (Helmholtz 2-D $\bbC$ bias).
        
        We start by investigating whether these algorithms are capable of solving the problem without preconditioning. Table~\ref{tab:results_universal}a shows the number of evaluations of the forward operator, $\A$, that is required to reduce the relative residue, $\norm{\A \EE - \source} / \norm{\source}$, to $10^{-3}$. 
        
        The GMRES algorithm converges for most problems. However, this method requires the storage of a growing number of Krylov sub-space vectors. It is therefore custom to restart GMRES after a small number of iterations. Here, we considered restarts every $20$ and every $5$ iterations, denoted by GMRES20 and GMRES5, respectively. More frequent restarts limit the algorithm's memory requirements, though also its efficiency as can be seen from the first two columns in Table~\ref{tab:results_universal}a. For two problems, even GMRES20 failed to converge within $30,000$ iterations, possibly due to numerical instabilities.
        
        The BiCGSTAB method is an attractive alternative to GMRES since it can allocate as few as $4$ temporary vectors~\cite{vanderVorst1992}. However, BiCGStab only converged in two out of the eight test cases. In all other cases, the algorithm stagnated, which is a known problem of this algorithm~\cite{Trefethen1997}.
        The fixed-point method can be seen to diverge independently of the step size $\fixedpointfactor$. This explains why this method finds little use in spite of its simplicity and minimal memory requirements. 

        We now compare these results to the results for the preconditioned system (Table~\ref{tab:results_universal}b).
        
        It can be seen that the universal split preconditioner of Eq.~\eqref{eq:preconditioner} ensures efficient convergence for all problems and algorithms. Not only does this preconditioner ensure that all methods converge to the required tolerance, the number of iterations is often reduced by more than an order of magnitude.
    
        Overall, GMRES20 required the lowest number of iterations. However, this advantage came at a cost of increased memory use because it must store the $20$ Krylov sub-space vectors. Also, GMRES has a considerable computation overhead related to the orthogonalization of the Krylov sub-space vectors~\cite{Trefethen1997}. For example, we found that the standard MATLAB implementation of GMRES20 spent only 5\% of the execution time on evaluating the preconditioned operator for the 1-D Helmholtz problem.
 
 In Table \ref{tab:results_universal}b, the methods with the shortest execution time are marked in bold-face, and is can be seen that GMRES is never the fastest method, even when it converges in the lowest number of iterations.
        
        For the fixed-point iterations, the step size $\fixedpointfactor$ plays an important role in the convergence behaviour, as expected from Theorem~\ref{th:convergence_rate_bounds}. It can be seen that the optimum values are found between $\fixedpointfactor=0.80$ and $1.00$.

        For wave problems, the universal split preconditioner generalises the convergent Born series method~\cite{Osnabrugge2016,Krueger2017,Vettenburg2019}. Even though the method has already been shown to outperform pseudo-spectral time domain and finite difference time domain methods in both speed and accuracy~\cite{Kaushik2020,Huang2020,Osnabrugge2021}, we can now conclude that the convergent Born series corresponds to the sub-optimal choices of $\fixedpointfactor=1$, $\norm{V}\rightarrow 1$, and a real bias in the $L+V$ splitting. The results for the Helmholtz problem with a high-contrast iron structure show that it is beneficial to choose a complex bias in cases that absorption is dominant. This optimisation improved the convergence speed of the fixed-point iteration by $30\%$ over the real-biased case. It should be noted that this optimisation makes the tuning of the step size, $\fixedpointfactor$, more critical for the fixed-point iteration.
        
        The fastest methods were either BiCGSTAB or the fixed-point iteration, depending on the specific problem. Although the fixed-point method is often disregarded for its tendency to diverge, it stands out for its simplicity, low overhead, and its minimal memory requirements. By ensuring the convergence of the fixed-point iteration, the universal split preconditioner can make it the method of choice for situations where memory usage is the limiting factor.

    \subsection{Comparison with alternative preconditioning}
        \noindent A large number of problem-specific preconditioners has been proposed in the literature. However, there are fewer preconditioners that are generally applicable to all linear problems. One that stands out here is the shift preconditioner proposed by Bai, Yin, and Su~\cite{Bai2006}. This preconditioner works for all accretive linear systems, uses fewer steps than Hermitian-skew-Hermitian splitting methods, and was shown to outperform other universal preconditioners, such as the incomplete triangular factorization (ILU) and incomplete Givens orthogonalization (IGO) preconditioners~\cite{Bai2001, Bai2009}. 
    
        The shift preconditioner can be used to solve strict-accretive systems, i.e.~systems where $\Re [\A] > 0$. It is defined by Bai et al.~as $\pre_\text{shift} = \frac12\left(\raw{\A} + \gamma \id \right)$, where $\gamma > 0$ is a positive scalar that must be chosen for optimal convergence~\cite{Bai2006}.
        Interestingly, this shift preconditioner is a  special case of the universal split preconditioner we defined in  Eq.~\eqref{eq:preconditioner}. Using $V = 0$, $L = A = \frac1\gamma \raw{\A}$, and $\fixedpointfactor = 2$ it can be seen that $\preinv\A = \preinv_\text{shift}\raw{\A}$.
        By Theorem~\ref{th:maintext-main}, convergence of the fixed-point iteration is guaranteed when $\fixedpointfactor < 2\Re\left[A^{-1}+B^{-1}\right]$. As $\fixedpointfactor=2$ and $\B=\id$, this only holds for the strict accretive systems assumed by Bai et al.

        Table~\ref{tab:results_shift} shows the number of evaluations with the shift-splitting preconditioner. The number of \emph{outer} iterations (Table~\ref{tab:results_shift}a) is sharply reduced compared to the non-preconditioned case, and the preconditioner makes the fixed-point convergent for most parameterisations ($\fixedpointfactor \le 0.9$). While the convergence properties are generally improved, in some cases GMRES20 or BiCGSTAB can be seen to stagnate (marked `s').

        The shift-splitting preconditioner, however, is computationally expensive to evaluate, since it requires the iterative solution of the inner problem $\frac12(A_0+\gamma I)\EE=\source$. This inner problem, itself, may need a pre-conditioner to be solved efficiently. In our tests, we preconditioned the inner problem with our universal split preconditioner since it ensures convergence of the inner iteration, and we solved it using BiCGSTAB.

        When taking into account these inner iterations (Table~\ref{tab:results_shift}b), it can be seen that when the the shift-split preconditioner is used,  roughly 10 to 1000 times more iterations are needed than when our universal split preconditioner is used.

    \section{Conclusion}
        \noindent The proposed universal split preconditioner is highly effective to solve a broad class of problems. It can be considered a generalisation of the shift-splitting preconditioner~\cite{Bai2006} to arbitrary splittings. This universality permits the use of a readily-invertible approximate system, such as the equivalent homogeneous problem, virtually eliminating the preconditioning cost. With Theorem~\ref{th:uniqueness}, we proved that Eq.~\eqref{eq:preconditioner} is the only possible universal split preconditioner that guarantees monotonic convergence of the fixed-point iteration. 

        We demonstrated that the universal split preconditioner is effective for a broad range of problems, including common inhomogeneous wave problems, diffusion problems, eigenvalue problems, and less-common problems such as the pantograph delay differential equation. We further showed that for non-accretive problems an augmented skew-Hermitian system can be constructed so that it can be solved using the universal split preconditioner.
        
        Even though for each of these applications problem-specific preconditioners may be constructed, the universal split preconditioner provides an effective generic approach that is straightforward to implement and analyse. As such, it can serve as a starting point for designing problem-specific or higher order approaches.
        
        The universal split preconditioner was essential to ensure convergence with gold-standard algorithms such as GMRES and BiCGSTAB, while outperforming the shift-splitting preconditioner. The universal split preconditioner guarantees monotonic convergence of the fixed-point iteration. Its minimal memory requirements, simplicity, and efficiency make it compelling for the most demanding problems.
        
    \section*{Acknowledgements}
        \noindent IMV is supported by the European Research Council under the European Union's Horizon 2020 Programme / ERC Grant Agreement $\text{n}^\circ$ 678919 and NWO Vidi grant $\text{n}^\circ$14879. TV is a UKRI Future Leaders Fellow supported by grant MR/S034900/1. IMV thanks Matthias Schlottbom for useful discussions. For the purpose of open access, the author(s) has applied a Creative Commons Attribution (CC BY) licence to any Author Accepted Manuscript version arising.
    
    \section*{Author contributions}
        \noindent Both authors developed the theory together, IMV implemented the wave and diffusion calculations, TV and IMV worked on the pantograph calculation. TV implemented the eigenmode calculations. Both authors wrote and reviewed the manuscript.
        
    \section*{Data and code availability}
        \noindent All presented data can be reproduced by the included code. The source code for all examples and visualisation is made publicly available as the AnySim library~\cite{github_anysim}.
        
    \section*{Appendix - Proofs of convergence and uniqueness}
    \noindent In this section, we first prove that our preconditioner stabilises \emph{any} accretive linear system in the sense that $\norm{\id-\preinv\A}<1$.    
    Next, we go one step further by showing that, up to a prefactor, our preconditioner is the \emph{only} preconditioner that universally stabilises all accretive linear systems with a single evaluation of the operator $(L+\id)^{-1}$. Finally, we place bounds on the condition number of the preconditioned system as well as on the convergence rate of the fixed-point iteration.

    In what follows, all operators act on a non-empty Hilbert space, $\cH$, and bounded operators have the complete Hilbert space, $\cH$, as their domain. This is not a severe restriction since, by the Hahn-Banach theorem, a bounded operator can always be extended to the full Hilbert space. Likewise, for any (potentially unbounded) invertible operator, we require that the domain of its (bounded) inverse is the whole Hilbert space.

    \subsection{Monotonic convergence of the preconditioned fixed-point iteration.~~~~~~~~~}
        \noindent To support the main theorem, we introduce two lemmas. 
    
        \begin{lemma}\label{th:lollipop_bound_general}
            Let $\Op$ be a bounded invertible linear operator. Then
            \begin{equation}
                \norm{\id-\Op}^2=1 + \sup_{\norm{\EE}=1}\left(\frac{1-2 \Re\inp{\EE, \Op^{-1}\EE}}{\norm{\Op^{-1}\EE}^2}\right)
                \label{eq:lemma-Rayleigh-condition-general}
            \end{equation}
            \begin{proof}
                Expand the definition of the operator norm squared $\norm{1-\Op}^2 =$
                \begin{align}
                    & \sup_{\source \ne \vec{0}}\frac{\norm{(\id-\Op)\source}^2}{\norm{\source}^2} = 1 + \sup_{\EE \ne \vec{0}}\frac{\norm{\EE}^2-2\Re\inp{\EE,\Op^{-1}\EE}}{\norm{\Op^{-1}\EE}^2},
                \end{align}
                where the substitution $\source=\Op^{-1}\EE$ is justified by the invertibility of $\Op$. The supremum of the fraction can be seen to be independent of $\norm{\EE}$ for $\EE \ne \vec{0}$, so we can choose $\norm{\EE}=1$, thereby completing the proof.
            \end{proof}
        \end{lemma}
    
        \begin{lemma}\label{th:lollipop_bound}
            For any invertible linear operator, $\Op$, the following conditions are equivalent
            \begin{equation}
                \norm{\id-\Op} < 1 \quad \Leftrightarrow \quad \Re\left[\Op^{-1}\right] > \frac12
            \end{equation}
            \begin{proof}
                From Lemma~\ref{th:lollipop_bound_general}, we find the equivalence $\norm{\id-\Op}^2 < 1 \; \Leftrightarrow$
                \begin{align}
                     \sup_{\norm{\EE}=1}\left(\frac{1 - 2 \Re\inp{\EE, \Op^{-1}\EE}}{\norm{\Op^{-1}\EE}^2}\right) &< 0 \quad  \Leftrightarrow\\ \inf_{\norm{\EE}=1}\left(2\Re\inp{\EE, \Op^{-1}\EE}-1\right) &> 0,
                \end{align}
                and using the definition Eq.~\eqref{eq:accretivity-condition} completes the proof.
            \end{proof}
        \end{lemma}
        
        \noindent The proof of the main theorem is underpinned by these lemmas. To state the main theorem, we define $\B\defeq \id-V$, so that $L+\id=\A+\B$ and $\preinv = \B(\A+\B)^{-1}$.
    
        \begin{theorem}\label{th:convergence-proof}
            Let $\A$, $\B$, and $\A+\B$ be invertible linear operators, and let $\fixedpointfactor > 0$; then
            \begin{equation}
                \norm{\id-\preinv \A}<1\quad\text{where}\quad \preinv\defeq \fixedpointfactor B(A+B)^{-1}, \label{eq:main-theorem}
            \end{equation}
            if and only if $\fixedpointfactor < 2\Re\left[\A^{-1}+\B^{-1}\right]$.
            \begin{proof}
                Lemma~\ref{th:lollipop_bound} shows that $\norm{1-\preinv \A}<1$ is equivalent to $\Re\left[A^{-1}\pre\right] > \frac12$. By substituting $\pre \defeq (A+B)B^{-1}\fixedpointfactor^{-1}$, the latter inequality can be rewritten as
                \begin{equation}
                    \Re\left[A^{-1}(A+B)B^{-1}\fixedpointfactor^{-1}\right] = \Re\left[A^{-1}+B^{-1}\right]\fixedpointfactor^{-1} > \frac12.
                \end{equation}
                This is equivalent to $\fixedpointfactor < 2\Re\left[\A^{-1}+\B^{-1}\right]$, thus completing the proof.
            \end{proof}
        \end{theorem}
    
        \begin{corollary}\label{th:main-corollary}
            Let $\A$ be an accretive invertible linear operator, and let $\B\defeq \id-V$, with $\norm{V}<1$; then, for any $\fixedpointfactor\in(0,1]$, it follows that $\preinv$ exists and $\norm{\id-\preinv \A}<1$.
        \begin{proof}
            First note that $\B$ is accretive, with $\Re{\left[\B\right]}\;\geq 1-\norm{V}$. Similarly, $\A + \B$ is accretive, with $\Re{\left[\A + \B\right]}\geq 1-\norm{V}$.
            The invertibility of $\B$ and $\A+\B$ follows (Lumer and Phillips \cite{Lumer1961}, Lemma~3.1). Hence, $\preinv \defeq \fixedpointfactor B(A+B)^{-1}$ can be constructed.

            By Lemma~\ref{th:lollipop_bound}, $\norm{\id-\B}<1$ implies that $2\Re\left[\B^{-1}\right] > 1 \ge \fixedpointfactor$. Since $\A$ is accretive, so is its inverse, $\Re\left[\A^{-1}\right] \ge 0$. Therefore, $2\Re\left[\A^{-1}+\B^{-1}\right] > \fixedpointfactor$. Substitution in Theorem~\ref{th:convergence-proof} completes the proof of the corollary.

            Note that since $\preinv$ is bounded, by the Hahn-Banach theorem its domain can be extended to the full Hilbert space, regardless of the domain of $\A$.

            Finally, defining $L \defeq \A - V$ we can substitute $A+B = L+\id$, giving the form used in Theorerem~\ref{th:maintext-main} in the main text.
        \end{proof}
        \end{corollary}
        \noindent As $M \defeq \id - \preinv \A$ is a contraction, it implies the monotonic convergence of Neumann series~\eqref{eq:neumann_series} and thus the preconditioned fixed-point iteration.

    \subsection{Uniqueness of the preconditioner}
        \noindent By Corollary~\ref{th:main-corollary}, our approach to construct a preconditioner is \emph{general} for all accretive $\A$ and all $\B$ with $\norm{\id-B}<1$. Central to our approach is the construction of the linear system $\A + \B$, for which we can efficiently evaluate the inverse $\left(\A + \B\right)^{-1}$. The resulting preconditioned system only requires a single evaluation of this inverse, i.e.~it is \emph{first order} in $\left(\A + \B\right)^{-1}$. 
    
        A natural question to ask is whether there exist other, perhaps more efficient, preconditioners that generally apply to all $\A$ and $\B$, and that are also first-order in $(\A+\B)^{-1}$. We show that the answer to this question is negative. Up to a constant scaling factor, the preconditioner proposed by Eq.~\eqref{eq:preconditioner} is unique.
        
        Specifically, we consider all preconditioners of the form
        \begin{equation}
            \preinv \defeq \beta_B (\A+\B)^{-1} \alpha_B + \gamma_B
            \label{eq:preconditioner-gammas}
        \end{equation}
        where $\alpha_\B$, $\beta_\B$, and $\gamma_\B$ may be \emph{any} invertible operator derived from $\B$ only (independent of $\A$). This includes any linear or non-linear function of $\B$, such as for example $\B^\ast$, $(\B^\ast \B)^\frac12$, or $1/\norm{B}$. In what follows we will show that for the preconditioner to be applicable to all accretive problems, $\gamma_B$ must be $0$, $\beta_\B$ must be proportional to $\B$, and $\alpha_\B$ must be constant. When $\beta_\B = \B$, the constant $\alpha_\B$ must be real and positive.
    
        Any \emph{general} preconditioner must also be applicable to the common case of separable Hilbert spaces, that is, Hilbert spaces with a countable orthonormal basis. Therefore, we can limit the proof of generality to that of separable Hilbert spaces.
        In these spaces, operators can be represented as (potentially infinite) matrices, thereby simplifying the proof.

        \begin{theorem}[Uniqueness]\label{th:uniqueness}
            Consider a separable Hilbert space, and let $\pre$ be an invertible linear operator of the form
            \begin{equation}
                \preinv \defeq \beta_B (\A+\B)^{-1} \alpha_B + \gamma_B\label{eq:preconditioner-gammas2}
            \end{equation}
            where $\alpha_B$, $\beta_B$, and $\gamma_B$ are invertible operators that are derived from $\B$ only.
            Then, $\norm{1-\preinv A}<1$ for all combinations of linear operators $\A$ and $\B$ with $\Re[\A]>0$ and $\norm{\id-B}<1$ if and only if
            \begin{equation}
                \preinv = \B (\A+\B)^{-1} \fixedpointfactor\label{eq:unique-preconditioner}
            \end{equation}
            with $\fixedpointfactor < 1$.
        \end{theorem}
        \begin{proof}
            The implication follows directly from Corollary~\ref{th:main-corollary}. To prove the converse statement, we eliminate all alternative choices for $\alpha_\B$, $\beta_\B$, and $\gamma_\B$ by presenting operators $\A$ for which we have the contradiction $\norm{\id - \preinv\A} \not< 1$.
            
            \textbf{Condition 1.} Take $\A=k$, with $k$ a positive real scalar. In the limit of large $k$, we have $\norm{\id-\preinv A\EE}= k\norm{\gamma_B \EE} + \bigO(1)$. This term can be arbitrarily large, violating the condition $\norm{\id - \preinv\A} < 1$, unless $\norm{\gamma_\B\EE} = 0$ for all $\EE$. So,  $\gamma_\B = 0$.
    
            \textbf{Condition 2.} Lemma~\ref{th:lollipop_bound} shows that $\norm{\id - \beta_B(\A+\B)^{-1} \alpha_B \A} < 1 \Leftrightarrow$
            \begin{align}
                \Re \left[\A^{-1}\alpha_B^{-1} (\A + \B) \beta_B^{-1} \right] > \frac12. \label{eq:condition-re-form}
            \end{align}
            We choose a basis in which $\beta_B^{-1}$ has all non-zero diagonal elements, e.g.~the Schur decomposition has the eigenvalues on the diagonal.
            The matrix $\A$ is chosen to equal the identity matrix, except for the element $A_{jj}=k^{-2}$. As test vector, we use the unit-length vector $\EE=(k \vec{e}_i + e^{\ii\phi} \vec{e}_j) / \sqrt{1+k^2}$, with $i \neq j$. Here, $\vec{e}_i$ denotes a vector with all components $0$, except for the $i$'th component which is $1$. Expanding Eq.~\eqref{eq:condition-re-form} for small $k$,
            \begin{multline}
                \Re \inp{\EE, \A^{-1}\alpha_B^{-1} (\A + \B) \beta_B^{-1} \EE} =\\ k^{-1}\,\Re\left[ \left(\alpha_B^{-1}\right)_{ij}\left(\beta_B^{-1}\right)_{jj} e^{\ii\phi} \right] + \bigO(1),\label{eq:counter-example2-expansion}
            \end{multline}
            shows that the dominant term contains an off-diagonal element of $\alpha_\B^{-1}$. Since we are free to choose the phase $\phi$, the real part of the leading term can be made negative, violating the condition $\norm{\id - \preinv\A} < 1$, unless all $\left(\alpha_B^{-1}\right)_{ij}=0$. Therefore, $\alpha_B$ must be diagonal.
        
            \textbf{Condition 3.}
            Condition 2 holds in all bases where the diagonal elements of $\beta_B$ are all nonzero. We make the proposition that $\alpha_B$ has at least two different diagonal elements. We can then consider the $2\times 2$ submatrix containing those elements and apply a change of basis using a Givens rotation
            \begin{align}
            R \defeq \begin{bmatrix}
                         \cos\theta & -\sin\theta \\
                         \sin\theta & \cos\theta
            \end{bmatrix}, &
            \quad \alpha_B\rightarrow
            R^{-1}
            \begin{bmatrix}
                \left(\alpha_B\right)_{ii} & 0 \\
                0 & \left(\alpha_B\right)_{jj}
            \end{bmatrix}R.
            \end{align}
            where $\theta\in(0,\pi)$ is chosen so that the diagonal elements of $R^{-1}\beta_B R$ are still nonzero. The resulting matrix $R^{-1}\alpha_B R$ can be verified to have 
            off-diagonal elements with the value $\frac12\left[\left(\alpha_B\right)_{jj}-\left(\alpha_B\right)_{ii}\right] \sin(2\theta)$, thus violating Condition 2. Therefore, we conclude that the proposition must be false and all diagonal elements of $\alpha_B$ must be equal, i.e. $\alpha_B$ is a scalar multiple of the identity matrix.

            \textbf{Condition 4.}
            We repeat the procedure above, now taking the leading term for $k\rightarrow \infty$.
            \begin{multline}
                \Re \inp{\EE, \alpha_B^{-1} (\id + A^{-1}\B)\beta_B^{-1} \EE} =\\ k\,\Re\left[  \alpha_B^{-1}e^{-\ii\phi}\left(B\beta_B^{-1}\right)_{ji}\right] + \bigO(1)\label{eq:counter-example3-expansion}.
            \end{multline}    
            
            Since this value must be positive for all $\phi$, and in all bases, we can follow the same reasoning as in Conditions 2 and 3, and conclude that $B\beta_B^{-1}$ must be a scalar multiple of the identity matrix. We choose $\beta_B=B$, and absorb the scalar constant in $\alpha_B$.
    
            \textbf{Condition 5.}
            Finally, we choose a scalar $A=k^{-1}e^{\ii\phi}$ with $\phi\in[-\pi/2,\pi/2]$, and $k$ large so that Eq.~\eqref{eq:condition-re-form} reduces to
            \begin{align}
                \Re \left[\fixedpointfactor_B^{-1} \left(\B^{-1} + \A^{-1}\right) \right] = k\Re \left[\fixedpointfactor_B^{-1} e^{-\ii\phi} \right] + \bigO(1)
            \end{align}
            The sign of the leading term can become negative for $\phi=\pm\pi/2$, unless $\fixedpointfactor$ is real and positive. Finally, by Theorem~\ref{th:convergence-proof} we require that $0<\fixedpointfactor < 2\Re[\B^{-1}] < 1$, leaving  Eq.~\eqref{eq:unique-preconditioner} as the only remaining choice for a preconditioner.
        \end{proof}

        We have now established our preconditioner as the only construct that works for \emph{all} accretive $\A$ and $\norm{1-\B}<1$. However, for \emph{specific} cases we may still have other preconditioners. Indeed, this is sometimes the case, for example in the trivial case that $\A$ is a scalar. 
        
        Important subclasses of linear systems are those with bounded operators $\A$, and those where $\A$ is positive definite.
        
        \begin{corollary}
            If we restrict $\A$ to bounded operators, Theorem~\ref{th:uniqueness} still holds: there is no other preconditioner that is generally applicable except for the one in Eq.~\eqref{eq:unique-preconditioner}.
            \begin{proof}
                In the proof of Theorem~\ref{th:uniqueness}, we only used bounded operators $\A$ as counter examples, so the proof directly translates to this case.
            \end{proof}
        \end{corollary}

        \begin{corollary}
            If we restrict $\A$ to positive definite operators, Theorem~\ref{th:uniqueness} still holds with the caveat that $\fixedpointfactor$ may be complex-valued in Eq.~\eqref{eq:unique-preconditioner}.
            \begin{proof}
                In the proof of Theorem~\ref{th:uniqueness}, we only used positive definite operators $\A$ as counter examples to derive Conditions~1-4. Only in the last step it was noted that $\fixedpointfactor$ should be real.
            \end{proof}
        \end{corollary}
        \noindent We found no advantage in choosing a non-real $\fixedpointfactor$ (see Theorem~\ref{th:convergence_rate_bounds} below).
    
        As a final note, a preconditioner based on the adjoint, $\left(\A^\ast + \B^\ast\right)^{-1}$, can also be ruled out. It is sufficient to replace $\A+\B$ with $\A^\ast+\B^\ast$ in Eq.~\eqref{eq:condition-re-form}, to verify that Condition 2 is always violated for $\A=k e^{\ii\phi}$, with $k\gg 1$, and some $\phi\in[-\pi/2,\pi/2]$.

    \subsection{Condition number.} 
        \noindent The relative condition number, $\kappa\left(\A\right)$, of a system, $\A$, is an important indicator of numerical stability and efficient convergence~\cite{Trefethen1997}. For bounded linear operators such as matrices, it is given by the product $\kappa\left(\A\right) = \norm{\A^{\vphantom{-1}}}\norm{\A^{-1}}$. In this section we show how the proposed preconditioner can markedly improve the condition number of a linear system. The relative condition number, $\kappa\left(\preinv\A\right)$, of the bounded and non-singular preconditioned system, can be estimated as 
        \begin{align}
            \kappa\left(\preinv\A\right) &\defeq
                \norm{\preinv \A}\norm{\A^{-1}\pre}\\&= \norm{\left(\A^{-1} + \B^{-1}\right)^{-1}}\norm{\A^{-1} + \B^{-1}} \\
                & \le \frac{\norm{\A^{-1} + \B^{-1}}}{\Re\left[\A^{-1}\right] + \Re\left[\B^{-1}\right]} \le \frac{\norm{\A^{-1}} + \norm{\B^{-1}}}{\Re[\B^{-1}]}.
        \end{align}
        Since the numerical range of $V$ is confined to a disk with radius $\norm{V}$ centred at the origin, we can use the bounds $\Re\left[\B^{-1}\right]\geq (1+\norm{V})^{-1}$ and $\norm{\B^{-1}}\leq (1-\norm{V})^{-1}$. Note that we can make $\norm{V}$ arbitrarily small by increasing the magnitude of the scaling prefactor $\scale$; however, this would increase $\norm{\A^{-1}}$ by the same factor. We therefore introduce the value $\condfactor \defeq \norm{\A^{-1}}\norm{V}$, which is invariant to such a scaling.
        \begin{align}
            \kappa\left(\preinv\A\right) &\le \frac{\norm{\A^{-1}} + (1 - \norm{V})^{-1}}{(1+\norm{V})^{-1}} \\
            & = \frac{S\norm{V}^{-1}+ (1 - \norm{V})^{-1}}{(1+\norm{V})^{-1}}
            \label{eq:S-substitution}
        \end{align}
        This equation has a minimum of
            \begin{equation}
                \kappa\left(\preinv\A\right) \leq (1+\sqrt{2\condfactor})^2
                \approx 2\condfactor \quad\text{at}\quad \norm{V}=\frac{\sqrt{\condfactor}}{\sqrt{\condfactor}+\sqrt{2}}
                \label{eq:condition-number-final-bound}
            \end{equation}
        where the approximation is valid for $\condfactor \gg \frac12$.
        We thus see that preconditioning improves the condition number for bounded systems when approximation error $\norm{\A - L} \ll \frac12 \norm{\A}$. Hence, for unbounded systems, any finite approximation error, $\A - L$, improves the conditioning of the system.
        
    \subsection{Convergence rate of the fixed-point iteration.}
        \noindent Algorithm~\ref{alg:basic} improves its estimate with consecutively smaller correction vectors. Each correction, $\Delta_\iterationcounter$, is a term in the Neumann series and related to its predecessor, $\Delta_{\iterationcounter-1}$, by $\left(\id - \preinv\A\right)\Delta_{\iterationcounter-1} \defeq \Delta_\iterationcounter$. This iteration converges monotonically at a rate bounded by $\norm{\id - \preinv\A}$. We now place a bound on this convergence rate, and determine the optimum value for $\fixedpointfactor$.
        
        \begin{theorem}[Convergence rate bounds]\label{th:convergence_rate_bounds}
            For an optimal choice of $\fixedpointfactor$, and overall scaling of the system, the convergence rate $\preconderr\defeq \norm{\id - \preinv \A}$ of the fixed-point iteration is bounded by
            \begin{equation}
                    \preconderr\leq\sqrt{1-\frac{1}{\left(1+\sqrt{2\condfactor}\right)^4}}\approx 1-\frac{1}{8 \condfactor^2}\;\;\text{ with }\condfactor \defeq \norm{\A^{-1}}\norm{V},
                    \label{eq:convergence-rate-theorem}
            \end{equation}
                where  the approximation is justified when $\condfactor \gg 1$. When $\A$ and $\B$ are self-adjoint, a tighter bound is given by
            \begin{equation}
                    \preconderr\leq 1-\frac{1}{1+\sqrt{2\condfactor}+\condfactor}\approx 1-\frac{1}{\condfactor}
                    \label{eq:convergence-rate-hermitian}
            \end{equation}
        \end{theorem}
        \begin{proof}
            Using Lemma~\ref{th:lollipop_bound_general}, and substituting  $\left(\preinv\A\right)^{-1}=\fixedpointfactor^{-1}\left(\A^{-1} + \B^{-1}\right)$ gives
            \begin{align}
                \preconderr^2&=\norm {\id - \preinv \A}^2\\&= \sup_{\norm{\EE} = 1}\left[1+\frac{\fixedpointfactor^2-2\fixedpointfactor\Re\inp{\EE, \left(\A^{-1}+\B^{-1}\right)\EE}}{\norm{\left(\A^{-1}+\B^{-1}\right)\EE}^2}\right].
                \label{eq:convergence-bounds-exact}
            \end{align}
            Since the fraction is negative by definition of $\fixedpointfactor$, we can obtain an upper bound,
            \begin{align}
                \preconderr^2&\leq  1+\frac{\fixedpointfactor^2-2\fixedpointfactor\Re \left(\A^{-1}+\B^{-1}\right)}{\norm{\A^{-1}+\B^{-1}}^2}\\
                &\leq  1+\frac{\fixedpointfactor^2-2\fixedpointfactor\Re\left[\B^{-1}\right]}{\left(\norm{\A^{-1}}+\norm{\B^{-1}}\right)^2}\label{eq:convergence-critical-estimate}
            \end{align}
            In the last step we used the fact that, $\Re[\A]\geq 0$, and thus also $\Re[\A^{-1}] \geq 0$.
            The tightest bound on the squared norm, $\preconderr^2$, is found for $\fixedpointfactor=\Re[\B^{-1}]\geq \frac12$:
            \begin{equation}
                \preconderr^2\leq  1-\left[\frac{\Re\left[ \B^{-1}\right]}{\norm{\A^{-1}}+\norm{\B^{-1}}}\right]^2.
                \label{eq:convergence-rate-condition-number}
            \end{equation}
    
            As in Eq.~\eqref{eq:S-substitution}, we consider that the spectrum of $V$ is confined to a disk with radius $\norm{V}$ centred at the origin, and substitute $\condfactor\defeq \norm{\A^{-1}}\norm{V}$.
            \begin{align}
                \preconderr^2&\leq 1-\left[\frac{\left(1+\norm{V}\right)^{-1}}{\norm{\A^{-1}}+\left(1-\norm{V}\right)^{-1}}\right]^2\\
                &=1-\left[\frac{\left(1+\norm{V}\right)^{-1}}{\condfactor / \norm{V}+\left(1-\norm{V}\right)^{-1}}\right]^2\label{eq:convergence-rate-rprime}.
            \end{align}
            $\norm{M}$ is minimised for the value of $\norm{V}$ given in Eq.~\eqref{eq:condition-number-final-bound}. Inserting this value gives the final result Eq.~\eqref{eq:convergence-rate-theorem}.
    
            In the special case that $M$ is Hermitian, e.g.~the case when both $\A$ and $\B$ are Hermitian, we tighten the bound further. By the spectral theorem~\cite{Hall2013quantum}, the bounded operator $\left(\A^{-1}+\B^{-1}\right)^{-1}$ is now unitarily equivalent to a multiplication operator with real spectrum~\cite{Hall2013quantum}. Therefore, we can simplify the convergence rate as
            \begin{align}
                \norm{M} & = \norm{1-\fixedpointfactor\left(\A^{-1}+\B^{-1}\right)^{-1}} = \sup_{\lambda} \abs{1 - \fixedpointfactor\lambda}.
            \end{align}
            where $\lambda \in [\lmin,\lmin]$, the spectrum of  $\left(\A^{-1}+\B^{-1}\right)^{-1}$. The value of $\fixedpointfactor$ can be chosen as $\fixedpointfactor = \frac{2}{\lmin + \lmax}$, to minimise the supremum as
            \begin{align}
                \norm{M} &  = \max\left(\abs{1 - \fixedpointfactor\lmin}, \abs{1 - \fixedpointfactor\lmax}\right) = \frac{\lmax - \lmin}{\lmax + \lmin},
            \end{align}
            which increases monotonically with the ratio $\lmax/\lmin$.
            An upper bound can be determined by considering that $\lmin \ge \left(\condfactor\norm{V}^{-1} + \left(1-\norm{V}\right)^{-1}\right)^{-1}$ and $\lmax \le 1+\norm{V}$. Substituting these bounds and reordering terms gives,
            \begin{equation}
                \preconderr\leq \frac{S+2\norm{V}^2-S\norm{V}^2}{S+2\norm{V}-S\norm{V}^2}
            \end{equation}
            Again, this function is minimised for the value of $\norm{V}$ given in Eq.~\eqref{eq:convergence-rate-hermitian}, resulting in Eq.~\eqref{eq:condition-number-final-bound}.
        \end{proof}
        \noindent The values for $\fixedpointfactor$ that optimise the worst-case convergence rate are given by $\fixedpointfactor=\Re\left[B^{-1}\right]\approx \frac12$ and $\fixedpointfactor=2/(\lmin + \lmax)\approx 1$ for the general case and the Hermitian case, respectively. However, note that the estimate made in Eq.~\eqref{eq:convergence-critical-estimate} is a pessimistic one because it assumes that the same vector $\EE$ that maximises $\norm{\A^{-1}\EE}$ has $\Re\inp{\EE, A^{-1}\EE} = 0$. We observed that this is rarely the case in practice. As can be seen in Table~\ref{tab:results_universal}, the optimum value of $\fixedpointfactor$ in the non-Hermitian case is typically closer to $1$.

    \bibliographystyle{elsarticle-num}
    \bibliography{bibliography}
\end{document}